\documentclass[12pt,a4paper]{amsart}
\usepackage{abstract}
\usepackage{amsmath}
\usepackage{amssymb}
\usepackage{amsthm}
\newtheorem{thm}{Theorem}
\newtheorem{lemm}[thm]{Lemma}
\newtheorem{prop}[thm]{Proposition}
\newcommand{\GL}{\mathrm{GL}}
\newcommand{\id}{\mathrm{id}}
\newcommand{\J}{\mathcal{J}}
\newcommand{\Mat}{\mathrm{Mat}}

\newcommand{\Z}{\mathbb{Z}}

\begin{document}
\begin{center}
  \textbf{Complementation in the Group of Units of Matrix Rings}\\ [15pt]
  Stewart Wilcox\\
\end{center}
\renewcommand{\abstractnamefont}{\normalfont\small}
\renewcommand{\abstractname}{ABSTRACT}
\begin{abstract}
Let $R$ be a ring with $1$ and $\J(R)$ its Jacobson radical. Then $1+\J(R)$ is a
normal subgroup of the group of units, $G(R)$. The existence of a complement to this
subgroup was explored in a paper by Coleman and Easdown; in particular the ring
$R=\Mat_n(\Z_{p^k})$ was considered. We prove the remaining cases to determine for
which $n$, $p$ and $k$ a complement exists in this ring.
\end{abstract}
\begin{center}\abstractnamefont{1. INTRODUCTION}\end{center}
If $R$ is a ring with $1$, let $G(R)$ denote its group of units. If $\psi:R\rightarrow S$
is a ring homomorphism which maps $1_R\mapsto1_S$, let $G(\psi):G(R)\rightarrow G(S)$
denote the corresponding group homomorphism. Denoting by $\J(R)$ the Jacobson radical of
$R$, it can be shown that $J(R)=1+\J(R)$ is a normal subgroup of $G(R)$. In \cite{CE}
results were found about the existence of a complement of $J(R)$. In particular these
were applied to partly classify the case when $R=\Mat_n(\Z_{p^k})$ for a prime $p$ and
integers $n,k\geq1$. The remaining values of $p$, $n$ and $k$ are considered in
Propositions \ref{no} and \ref{yes} to give the following results.
\begin{thm}
Let $R=\Mat_n(\Z_{p^k})$. Then $J(R)$ has a complement in $G(R)$ exactly when
\begin{itemize}
\item $k=1$, or
\item $k>1$ and $p=2$ with $n\leq 3$, or
\item $k>1$ and $p=3$ with $n\leq 2$, or
\item $k>1$ and $p>3$ with $n=1$.
\end{itemize}
\end{thm}
When $k=1$ the subgroup $J(R)$ is trivial, and so complemented. Theorems 4.3 and 4.5 of
\cite{CE} can be summarised as
\begin{thm}[Coleman-Easdown]
Define $R$ as above. If $p=2$ or $3$ and $n=2$, then $J(R)$ has a complement in
$G(R)$. If $p>3$, $n\geq2$ and $k\geq2$ then $J(R)$ has no complement.
\end{thm}
It is well known (see, for example, Theorem 11.05 of \cite{mac}) that there exists
$a\in\Z_{p^k}$ with order $p-1$. The subgroup generated by $a$ then complements
$1+p\Z_{p^k}$ in $G(\Z_{p^k})$, so a complement always exists when $n=1$. Thus it remains
to prove existence when $p=2$ with $n=3$, and disprove existence when $p=2$ with $n\geq4$
and $p=3$ with $n\geq3$. Before proving Propositions \ref{no} and \ref{yes}, we make some
preliminary observations. Since $\Z_{p^k}$ is local, clearly $\J(\Z_{p^k})=p\Z_{p^k}$ and
$\Z_{p^k}/\J(\Z_{p^k})\cong\Z_p$. Let $\phi_k:\Z_{p^k}\twoheadrightarrow\Z_p$ be the
natural surjection. From Theorem 30.1 of \cite{SZASZ}, we have
\[
  \J(\Mat_n(S))=\Mat_n(\J(S))
\]
for any ring $S$. In particular with $R=\Mat_n(\Z_{p^k})$ as above,
\[
  \J(R)=\Mat_n(\J(\Z_{p^k}))=\Mat_n(p\Z_{p^k})
\]
so that
\[
  R/\J(R)\cong\Mat_n(\Z_p)
\]
Let $\psi_{n,k}:R\twoheadrightarrow\Mat_n(\Z_p)$ be the corresponding surjection,
which is induced by $\phi_k$ in the obvious way. Then $G(\psi_{n,k})$ is surjective
with kernel $J(R)$. Thus $J(R)$ is complemented in $G(R)$ if and only if there exists
a group homomorphism $\theta:\GL_n(\Z_p)\rightarrow G(R)$ with
$G(\psi_{n,k})\theta=\id_{\GL_n(\Z_p)}$.
\begin{center}\abstractnamefont{2. NONEXISTENCE}\end{center}
We first reduce to the case $k=2$ and $n$ minimal.
\begin{lemm}
Assume $k>1$ and let $R=\Mat_n(\Z_{p^k})$ as above. Pick any $m\leq n$. If $J(R)$ has
a complement in $G(R)$, then $J(S)$ has a complement in $G(S)$ where
$S=\Mat_m(\Z_{p^2})$.
\end{lemm}
\begin{proof}
Since $J(R)$ has a complement, the discussion of the previous section shows that
there exists $\theta':\GL_n(\Z_p)\rightarrow G(R)$ with
\[
  G(\psi_{n,k})\theta'=\id_{\GL_n(\Z_p)}
\]
We have a ring homomorphism $\lambda:\Z_{p^k}\twoheadrightarrow\Z_{p^2}$ satisfying
$\phi_2\lambda=\phi_k$. Then $\lambda$ induces
$\mu:\Mat_n(\Z_{p^k})\twoheadrightarrow\Mat_n(\Z_{p^2})$, which satisfies
$\psi_{n,2}\mu=\psi_{n,k}$. Thus
\[
  \id_{\GL_n(\Z_p)}=G(\psi_{n,k})\theta'=G(\psi_{n,2})G(\mu)\theta'
    =G(\psi_{n,2})\theta
\]
where $\theta=G(\mu)\theta'$. Denote $\psi_{n,2}$ by $\psi_n$. Let $H\leq\GL_n(\Z_p)$
be the subgroup consisting of all matrices of the form
\[
  \begin{pmatrix}A&0\\0&I_{n-m}\end{pmatrix}
\]
where $A\in\GL_m(\Z_p)$ and $I_{n-m}$ is the identity matrix of size $n-m$. Then
$H'=\psi_n^{-1}(H)$ contains $\theta(H)$, and consists of those invertible matrices
$A=(a_{ij})$ which satisfy $a_{ij}-\delta_{ij}\in p\Z_{p^2}$ whenever $i>m$ or $j>m$.
Pick elements $A=(a_{ij})$ and $B=(b_{ij})$ of $H'$, and assume $i,j\leq m$ but $l>m$.
Clearly $\delta_{il}=\delta_{lj}=0$ so that $a_{il},b_{lj}\in p\Z_{p^2}$. Hence
$a_{il}b_{lj}=0$, so that for $i,j\leq m$ we have
\[
  (ab)_{ij}=\sum_{l=1}^na_{il}b_{lj}=\sum_{l=1}^m a_{il}b_{lj}
\]
Thus mapping the matrix $A=(a_{ij})_{1\leq i,j\leq n}\in H'$ to $(a_{ij})_{1\leq
i,j\leq m}$ gives a homomorphism $\nu:H'\rightarrow G(\Mat_m(\Z_{p^2}))$. But there is
an obvious isomorphism $\kappa:\GL_m(\Z_p)\rightarrow H$ and this satisfies
\[
  \psi_m\nu\theta\kappa=\id_{\GL_m(\Z_p)}
\]
noting that the image of $\theta\kappa$ lies in the domain of $\nu$. Since
$\theta_1=\nu\theta\kappa$ is a homomorphism, the result follows.
\end{proof}
\begin{prop}\label{no}
Assume $k>1$ and define $R$ as above. If $p=2$ with $n\geq4$, or $p=3$ with $n\geq3$,
then $J(R)$ has no complement in $G(R)$.
\end{prop}
\begin{proof}
By the previous Lemma, we may assume that $k=2$, and that $n=4$ when $p=2$ and $n=3$ when
$p=3$. First take the $p=3$ case, and consider the following two elements of
$\GL_3(\Z_3)$
\[
  \alpha=\begin{pmatrix}1&2&0\\0&1&0\\0&0&1\end{pmatrix}
  \hspace{10mm}\text{and}\hspace{10mm}
  \beta=\begin{pmatrix}1&0&2\\0&1&0\\0&0&1\end{pmatrix}
\]
It is easy to verify that $\alpha^3=1$ and $\alpha\beta=\beta\alpha$. Since
$\psi_3\theta=\id$ we may write
\begin{eqnarray*}
  \theta(\alpha)&=&\begin{pmatrix}
    3a+1&3b+2&3c\\3d&3e+1&3f\\3g&3h&3i+1
  \end{pmatrix}\\
  \theta(\beta)&=&\begin{pmatrix}
    3p+1&3q&3r+2\\3s&3t+1&3u\\3v&3w&3x+1
  \end{pmatrix}
\end{eqnarray*}
where all variables are integers. Then entry $(1,2)$ of $\theta(\alpha^3)=\theta(1)$
gives $d=1\pmod{3}$, while entry $(2,3)$ of $\theta(\alpha\beta)=\theta(\beta\alpha)$
gives $d=0\pmod{3}$, clearly a contradiction. Now assume $p=2$, and consider the
following two elements of $\GL_4(\Z_2)$
\[
  \alpha=\begin{pmatrix}1&0&1&0\\0&1&1&1\\0&0&1&0\\0&0&0&1\end{pmatrix}
  \hspace{10mm}\text{and}\hspace{10mm}
  \beta=\begin{pmatrix}1&0&0&1\\0&1&1&0\\0&0&1&0\\0&0&0&1\end{pmatrix}
\]
It is easy to verify that $\alpha^2=\beta^2=1$ and $\alpha\beta=\beta\alpha$. Since
$\psi_4\theta=\id$ we may write
\[
  \theta(\alpha)=\begin{pmatrix}
    2a+1&2b&2c+1&2d\\
    2e&2f+1&2g+1&2h+1\\
    2i&2j&2k+1&2l\\
    2m&2n&2o&2p+1
  \end{pmatrix}
\]
and
\[
  \theta(\beta)=\begin{pmatrix}
    2q+1&2r&2s&2t+1\\
    2u&2v+1&2w+1&2x\\
    2y&2z&2A+1&2B\\
    2C&2D&2E&2F+1
  \end{pmatrix}
\]
where again all variables are integers. After a lengthy calculation, from entries
$(1,3),\;(1,4)$ and $(2,4)$ of $\theta(\alpha^2)=1$ we obtain
\begin{eqnarray*}
  a+b+k&=&1\pmod{2}\\
  b+l&=&0\pmod{2}\\
  f+l+p&=&1\pmod{2}
\end{eqnarray*}
Similarly from entries $(1,3),\;(2,3)$ and $(2,4)$ of $\theta(\beta^2)=1$ we obtain
\begin{eqnarray*}
  E+r&=&0\pmod{2}\\
  A+v&=&1\pmod{2}\\
  B+u&=&0\pmod{2}
\end{eqnarray*}
Finally comparing entries $(1,4)$ and $(2,3)$ of
$\theta(\alpha\beta)=\theta(\beta\alpha)$,
\[
  a+B+p+r=0\pmod{2}\hspace{6mm}\text{and}\hspace{6mm}
  A+E+f+k+u+v=0\pmod{2}
\]
Summing the above $8$ equations gives $0=1\pmod{2}$, and we have the required
contradiction.
\end{proof}
\begin{center}\abstractnamefont{3. EXISTENCE}\end{center}
\begin{prop}\label{yes}
Assume $k>1$ and define $R$ as above. If $p=2$ and $n=3$ then $J(R)$ has a complement
in $G(R)$.
\end{prop}
\begin{proof}
The group $GL_3(\Z_2)$ has the following presentation, which can be easily verified using a standard
computer algebra package such as MAGMA \cite{MAGMA},
\[
  GL_3(\Z_2)=\langle\alpha,\beta\mid\alpha^2=\beta^3=(\alpha\beta)^7
    =(\alpha\beta\alpha\beta^{-1})^4=1\rangle
\]
where
\[
  \alpha=\begin{pmatrix}
    1&1&0\\
    0&1&0\\
    0&0&1
  \end{pmatrix}
  \hspace{10mm}\text{and}\hspace{10mm}
  \beta=\begin{pmatrix}
    0&0&1\\
    1&0&0\\
    0&1&0
  \end{pmatrix}
\]
%First we will prove by induction on $k\geq1$ that there exists $a\in\Z$ with
%\[
%  a=1\pmod{2}
%  \hspace{10mm}\text{and}\hspace{10mm}
%  a^2+a+2=0\pmod{2^k}
%\]
%The case $k=1$ is clear by taking $a=1$. Assume such $a$ has been chosen for some
%$k\geq1$, and let $l=a^2+a+2$. By assumption $l=0\pmod{2^k}$, so since $k\geq1$ we
%have
%\[
%  2l=0=l^2\pmod{2^{k+1}}
%\]
%If we set $b=a-l$, then $b=a=1\pmod{2}$ since $k\geq1$, and
%\[
%  b^2+b+2=a^2-2la+l^2+a-l+2=l-l=0\pmod{2^{k+1}}
%\]
%as required, so the claim is proven. We now construct a homomorphism
%$\theta:\GL_3(\Z_2)\rightarrow\GL_3(\Z_{2^k})$ such that
%\begin{equation}\label{id}
%  G(\psi_{3,k})\theta=\id_{\GL_3(\Z_2)}
%\end{equation}
%Choose $a$ with $a=1\pmod{2}$ and $a^2+a+2=0\pmod{2^k}$, and let
We will construct a homomorphism $\theta:\GL_3(\Z_2)\rightarrow\GL_3(\Z_{2^k})$ such that
\begin{equation}\label{id}
  G(\psi_{3,k})\theta=\id_{\GL_3(\Z_2)}
\end{equation}
Now there exists $a$ with $a=1\pmod{2}$ and $a^2+a+2=0\pmod{2^k}$ by Hensel's Lemma. Define
$\bar{\alpha},\bar{\beta}\in\GL_3(\Z_{2^k})$ by
\[
  \bar{\alpha}=\begin{pmatrix}
    1&a&-a-1\\
    0&-1&0\\
    0&0&-1
  \end{pmatrix}
  \hspace{5mm}\text{and}\hspace{5mm}
  \bar{\beta}=\begin{pmatrix}
    0&0&1\\
    1&0&0\\
    0&1&0
  \end{pmatrix}
\]
It is easily verified using $a^2+a+2=0\pmod{2^k}$ that
\[
  \bar{\alpha}^2=1\hspace{5mm}
  \bar{\beta}^3=1\hspace{5mm}
  (\bar{\alpha}\bar{\beta})^7=1\hspace{5mm}
  (\bar{\alpha}\bar{\beta}\bar{\alpha}\bar{\beta}^{-1})^4=1
\]
We can then define $\theta$ by $\theta(\alpha)=\bar{\alpha}$ and
$\theta(\beta)=\bar{\beta}$. Then (\ref{id}) holds for $\alpha$ and $\beta$, since
$a=1\pmod{2}$. But $\alpha$ and $\beta$ generate $\GL_3(\Z_2)$, so (\ref{id}) holds.
The result then follows by the observations of Section 1.
\end{proof}
\begin{center}\abstractnamefont{ACKNOWLEDGEMENTS}\end{center}
This work was completed while the author was a recipient of a vacation scholarship
from the School of Mathematics and Statistics, University of Sydney.
% \bibliography{GSGpaper}
% \bibliographystyle{plain}

\end{document}